\documentclass[twoside]{article}

\usepackage{amsmath,amsthm,amssymb}

\usepackage{graphicx}

\usepackage[text={12.5cm,19cm},centering,paperwidth=17cm,paperheight=24cm]{geometry}

\def\titlerunning#1{\gdef\titrun{#1}}

\makeatletter
\def\author#1{\gdef\autrun{\def\and{\unskip, }#1}\gdef\@author{#1}}
\def\address#1{{\def\and{\\\hspace*{15.6pt}}\renewcommand{\thefootnote}{}\footnote{#1}}\markboth{\autrun}{\titrun}}
\makeatother

\def\email#1{email: \href{mailto:#1}{#1} }
\def\subjclass#1{\par\bigskip\noindent\textbf{Mathematics Subject Classification 2020.} #1}
\def\keywords#1{\par\smallskip\noindent\textbf{Keywords.} #1}

\newenvironment{acknowledgments}{\bigskip\small\noindent\textit{Acknowledgments.}}{\par}

\newtheorem{thm}{Theorem}[section]
\newtheorem{cor}[thm]{Corollary}
\newtheorem{lem}[thm]{Lemma}

\newtheorem{prop}[thm]{Proposition}

\theoremstyle{definition}
\newtheorem{defin}[thm]{Definition}
\newtheorem{exa}[thm]{Example}

\newtheorem*{rem}{Remark}

\numberwithin{equation}{section}

\def\Z{\mathbb{Z}}

\usepackage[hyperfootnotes=false,colorlinks=true,allcolors=blue]{hyperref}

\begin{document}

\titlerunning{Universality and undecidability in Euler and Reeb flows}

\title{\textbf{Looking at Euler flows through a contact mirror: Universality and undecidability}}

\author{Robert Cardona \and Eva Miranda \and Daniel Peralta-Salas}

\date{}

\maketitle

\address{R. Cardona: Laboratory of Geometry and Dynamical Systems, Department of Mathematics, Universitat Polit\`{e}cnica de Catalunya and BGSMath Barcelona Graduate School of
Mathematics,  Avinguda del Doctor Mara\~{n}on 44-50, 08028 , Barcelona; \email{robert.cardona@upc.edu} \and E. Miranda: Laboratory of Geometry and Dynamical Systems $\&$ Institut de Matem\`atiques de la UPC-BarcelonaTech (IMTech),  Universitat Polit\`{e}cnica de Catalunya, Avinguda del Doctor Mara\~{n}on 44-50, 08028 , Barcelona $\&$ CRM Centre de Recerca Matem\`{a}tica $\&$ IMCCE, CNRS-UMR8028, Observatoire de Paris, PSL University, Sorbonne
Universit\'{e}; \email{eva.miranda@upc.edu} \and D. Peralta-Salas: Instituto de Ciencias Matem\'aticas-ICMAT, C/ Nicol\'{a}s Cabrera, nº 13-15 Campus de Cantoblanco, Universidad Aut\'{o}noma de Madrid,
28049 Madrid, Spain; \email{dperalta@icmat.es}}


\begin{abstract}

The dynamics of an inviscid and incompressible fluid flow on a Riemannian  manifold  is  governed  by  the  Euler  equations. In recent papers \cite{CMP1, CMP2, CMPP1, CMPP2} several unknown facets of the Euler flows have been discovered, including universality properties of the stationary solutions to the Euler equations. The study of these universality features was suggested by Tao as a novel way to address the problem of global existence for Euler and Navier-Stokes \cite{TNat}. Universality of the Euler equations was proved in \cite{CMPP1} for stationary solutions using a contact mirror which reflects a Beltrami flow as a Reeb vector field. This contact mirror permits the use of advanced geometric techniques in fluid dynamics. On the other hand, motivated by Tao's approach relating Turing machines to Navier-Stokes equations, a Turing complete stationary Euler solution on a Riemannian $3$-dimensional sphere was constructed in \cite{CMPP2}. Since the Turing completeness of a vector field can be characterized in terms of the halting problem, which is known to be undecidable \cite{Turing},  a striking consequence of this fact is that a Turing complete Euler flow exhibits undecidable particle paths~\cite{CMPP2}.
In this article, we give a panoramic overview of this fascinating subject, and go one step further in investigating the undecidability of different dynamical properties of Turing complete flows. In particular, we show that variations of~\cite{CMPP2} allow us to construct a stationary Euler flow of Beltrami type (and, via the contact mirror, a Reeb vector field) for which it is undecidable to determine whether its orbits through an explicit set of points are periodic.

%
\end{abstract}

\section{Introduction}

Back in 1936 Turing faced a fundamental question which had been driving the attention of many mathematicians since the 1920's:  \emph{Is there an answer for the decision problem for first-order logics?}. A decision problem can be posed as a \emph{yes/no} question depending on the input values.  \emph{Decidability} is the problem of the existence of an effective method, a test or automatic procedure to know whether certain premises entail certain conclusions. The halting problem is one of the first decision problems which was proved to be undecidable. Indeed, Alan Turing proved  that a general algorithm that solves the halting problem cannot exist (for all possible program-input pairs). In doing so, he, fortuitously, invented the basic model of modern digital computers, the so-called Turing machine.

The undecidability of the halting problem yields a cascade of related questions: \emph{What kind of physics might be non-computational?} (Penrose~\cite{penrose})
 \emph{Is hydrodynamics capable of performing computations?} (Moore~\cite{Mo}). Given the Hamiltonian of a quantum many-body system, does there exist an algorithm to check whether it has a spectral gap? (this is known as \emph{the spectral gap problem}, recently proved to be undecidable~\cite{perez}). And last but not least, \emph{can a mechanical system (including a fluid flow) simulate a universal Turing machine}? (Tao~\cite{T1,T2, T3}).

Surprisingly, this last question is connected with the regularity of the Navier-Stokes equations~\cite{T0},  one of the unsolved problems in the Clay's list of problems for the Millennium . In~\cite{TNat} Tao speculated on a relation between a potential blow-up of the Navier-Stokes equations, Turing completeness and fluid computation. This is part of a more general programme he launched in \cite{T0,T1,TNat} to address the global existence problem for Euler and Navier-Stokes based on the concept of \emph{universality}.  Inspired by this proposal, in \cite{CMPP1} we showed that the stationary Euler equations exhibit several universality features, in the sense that, any non-autonomous flow on a compact manifold can  be  extended  to  a  smooth  stationary  solution  of  the Euler equations on a Riemannian manifold of possibly higher dimension.  As a corollary, we established the  Turing  completeness  of  the  steady  Euler  flows  on  a  $17$-dimensional sphere \cite{CMPP1}. It is then natural to ask: can this dimensional bound be improved?

We solved this problem affirmatively in~\cite{CMPP2} constructing stationary solutions of the Euler equations on a Riemannian 3-dimensional sphere that can simulate any Turing machine (i.e., they are Turing complete). In particular, these solutions exhibit undecidable paths in the sense that there are constructible points for which it is not possible to decide whether their associated trajectories will intersect a certain (explicit) open set or not. The type of flows that we considered are Beltrami fields, a particularly relevant class of stationary solutions. Our game plan combines the computational power of symbolic dynamics with techniques from contact topology. Contact topology enters into the scene because Beltrami fields correspond to Reeb flows under a contact mirror unveiled by Sullivan, Etnyre and Ghrist more than two decades ago.
The contact mirror thus reflects a problem in Fluid Dynamics as a problem in contact geometry and back.
 %


The existence of Turing complete Euler flows gives rise to new questions concerning undecidability of different dynamical properties. One of the potential problems to consider is that of periodic orbits: Ever, at least since the work of Poincar\'e \cite{poincare}, periodic orbits are known to be one of the major tools to understand the dynamics of Hamiltonian systems. Even though, not every Hamiltonian system admits periodic orbits, the
Weinstein conjecture asserts that under some topological (compact)  and geometrical (contact) conditions on the manifold, Reeb vector fields admit at least one periodic orbit. The Weinstein conjecture is known to be true in dimension 3, so using our contact mirror we can conclude that the Turing complete Reeb flow we constructed in \cite{CMPP2} has at least one periodic orbit (in fact, in our construction the Reeb vector field coincides with a Hopf field in the complement of a certain solid torus, so it has infinitely many periodic orbits). It is then natural to ask if for every point of the sphere it is possible to decide whether its corresponding orbit will be closed or not. We shall see in this article that such a decision problem has no answer. The undecidability of other dynamical properties of Reeb flows will be also discussed. In view of G\"{o}del's incompleteness theorems, undecidability of such properties of dynamical systems seems to be an unsurmountable obstacle no matter what systems of axioms are considered.

Our goal in this article is to give an overview of this exciting area of research. Let us summarize the contents of this work. Next, in this Introduction, we present the Euler equations and the Beltrami fields on Riemannian manifolds, in Section~\ref{SS:Euler}, and the connection between contact geometry and hydrodynamics (in particular, between Beltrami fields and Reeb flows), in Section~\ref{SS:contact}. In Section~\ref{S.embedding}, following~\cite{CMPP1}, we introduce the theory of Reeb embeddings and their flexibility (in the form of a new $h$-principle), and apply it to prove several universality features of the stationary Euler flows in high dimensions. The construction of a Turing complete Reeb field on a $3$-dimensional sphere~\cite{CMPP2} is presented in Section~\ref{S.Turing}; as a novel feature, we show how variations of this result allow us to prove the existence of Reeb fields exhibiting different undecidable dynamical properties, including periodic orbits. Finally, in Section~\ref{S.future} we recall the main theorem of~\cite{CMP2} establishing the existence of Turing complete time-dependent solutions to the Euler equations (on compact Riemannian manifolds of very high dimension), and discuss the implications of our results regarding computability with the Navier-Stokes equations.



%
%

\subsection{The Euler equations on Riemannian manifolds}\label{SS:Euler}

The Euler equations describe the dynamics of an incompressible fluid flow without viscosity. Even if they are classically considered on $\mathbb{R}^3$, they can be formulated on any $n$-dimensional Riemannian manifold $(M,g)$, $n\geq 2$ (for an introduction to the geometric aspects of hydrodynamics see~\cite{AK,peralta}). The equations can be written as:
\begin{equation*}
\begin{cases}
\frac{\partial}{\partial t} X + \nabla_X X &= -\nabla p\,, \\
\operatorname{div}X=0\,,
\end{cases}
\end{equation*}
where $p$ stands for the hydrodynamic pressure and $X$ is the velocity field of the fluid (a non-autonomous vector field on $M$). Here $\nabla_X X$ denotes the covariant derivative of $X$ along $X$. A solution to the Euler equations is called stationary whenever $X$ does not depend on time, i.e., $\frac{\partial}{\partial t} X=0$, and it models a fluid flow in equilibrium.

This extension of the Euler equations to high dimensional manifolds turns out to be very useful to show that the steady and time-dependent Euler flows exhibit remarkable dynamical~\cite{CMPP1} (see also~\cite{T2, T3, TdL}), computational~\cite{CMP2} or topological~\cite{C1} universality features.

\textbf{{ A short comprehensive dictionary}:}
\begin{itemize}

\item A volume-preserving (autonomous) vector field $X$ on $M$
is {Eulerisable}~\cite{PRT} if there exists a Riemannian metric $g$ on $M$ compatible with the volume form, such that $X$ satisfies the stationary Euler equations on $(M,g)$
\begin{equation}\label{euler_stat}
\nabla_XX=-\nabla p\,, \qquad \operatorname{div}X=0
\end{equation}
for some pressure function $p$.

\item A divergence-free vector field $X$ on an odd-dimensional manifold $(M,g)$ of dimension $n=2m+1$ is Beltrami if
$$ \operatorname{curl}X=fX\,,$$
for some factor $f \in C^\infty (M)$. The curl of $X$ is defined as the unique vector field that satisfies the equation
\[
i_{\operatorname{curl}X}\mu=(dX^\flat)^m
\]
where $\mu$ is the Riemannian volume form and $^\flat$ stands for the musical isomorphism associated to the metric $g$.  The classical Hopf fields on the round sphere $\mathbb S^{2m+1}$ and the ABC flows on the flat 3-torus $\mathbb T^3$ are examples of Beltrami fields.
\end{itemize}

\subsection{Contact hydrodynamics}\label{SS:contact}

Let $M^{2m+1}$ be an odd-dimensional manifold equipped with a hyperplane distribution $\xi$. Assume that there is a one-form $\alpha \in \Omega^1(M)$ with $\ker \alpha = \xi$ and $\alpha \wedge (d\alpha)^m >0$ everywhere. Then we say that $(M^{2m+1},\xi)$ is a (cooriented) contact manifold.

The one-form $\alpha$ is called a contact form. Of course, the contact structure $\xi$ does not depend on the choice of the defining one-form $\alpha$. It is well known that $d\alpha$ induces a symplectic structure on the hyperplane distribution $\xi$ (of even dimension $2m$). The unique Reeb vector field $R$ associated to a given contact form $\alpha$ is defined by the equations
\begin{equation} \label{eq:Reeb}
\iota_R\alpha=1\,, \qquad \iota_Rd\alpha =0\,.
\end{equation}


We will now explain the connection between contact geometry and hydrodynamics. In order to understand this remarkable correspondence it is convenient to rewrite the Euler equations in a \emph{dual language}. Duality is given by contraction with the Riemannian metric $g$. With the one-form $\alpha$ defined as
$\alpha:=X^\flat$ and the Bernoulli function as $B:=p+\frac{1}{2}g(X,X)$, the steady Euler equations can be equivalently formulated as
\begin{equation*}
\begin{cases}
\iota_X d\alpha=-dB\,, \\
d\iota_X\mu=0\,,
\end{cases}
\end{equation*}
where $\mu$ is the Riemannian volume form.

Observe that:
\begin{itemize}

\item The equation $ \operatorname{curl}X=fX, \text{ with } f \in C^\infty (M) $, satisfied by a Beltrami vector field on an odd-dimensional manifold, can be equivalently written as {$ (d\alpha)^m = f \iota_X \mu\,. $} Assume that $X$ is rotational, i.e., $f>0$, then if $X$ does not vanish on $M$ we infer that

$$\alpha \wedge (d\alpha)^m= f \alpha \wedge \iota_X \mu > 0\,,$$

thus proving that $\alpha$ defines a contact structure on $M$.

\item Obviously, $X$ satisfies {$\iota_X (d\alpha)^m= f\iota_X \iota_X \mu=0$}. Therefore, since $\alpha\wedge (d\alpha)^m>0$, it is easy to conclude that $X \in \ker d\alpha$, and hence it is a reparametrization of the Reeb vector field $R$ by the function $\alpha(X)= g(X,X)$, i.e., $R=\frac{X}{\alpha(X)}$.
\end{itemize}

These observations prove one of the implications of the following theorem, which is due to Etnyre and Ghrist~\cite{EG}.

\begin{thm}\label{thm:EGcorres}
Let $M$ be a Riemannian odd-dimensional manifold. Any smooth, nonsingular rotational
Beltrami field on $M$ is a Reeb-like field for some contact form on $M$. Conversely, given a
contact form $\alpha$ on $M$ with Reeb field $X$, any nonzero rescaling of $X$ is a smooth, nonsingular
rotational Beltrami field for some Riemannian metric on $M$.
%
\end{thm}

\begin{rem}
The original proof by Etnyre and Ghrist is for three dimensional manifolds. The fact that the correspondence holds on any odd-dimensional manifold was detailed in~\cite{CMPP1}. See also~\cite{CMP1} for an extension of this result to $b$-manifolds.
\end{rem}
%
%
%
%
%
%
%
%

\section{Embedding dynamics into Reeb flows}\label{S.embedding}

In~\cite{CMPP1}, we studied several universality features of the stationary Euler equations. In view of the correspondence established in Theorem~\ref{thm:EGcorres}, we can reformulate the question of embedding dynamics into steady Euler flows in terms of Reeb flows. Let us fix a nonvanishing vector field $X$ on a compact manifold $N$ and some compact contact manifold $(M,\xi)$ of dimensions $n\leq m$, respectively. The question we answer in this section is the following: Can we give sufficient conditions for the existence of an embedding $e:N\hookrightarrow M$ and a contact form $\alpha \in \Omega^1(M)$ defining $\xi$ such that the Reeb field $R$ satisfies $e_*X=R|_{e(N)}$?

\subsection{Flexibility of Reeb embeddings}
We will address the question above using a classical framework for flexibility problems in contact geometry: the homotopy principle. The world of contact geometry exhibits a lot of flexibility which reduces geometrical problems to their associated purely homotopical algebraic problems. The pioneering work of Gromov~\cite{G} showed that this approach is extremely fruitful for symplectic and contact geometrical problems. Some of Gromov's results in contact geometry were generalized in~\cite{BEM} when the ambient manifold is closed and the contact structure is ``overtwisted". We will not introduce this notion here, the only thing that we need in our discussion is that being ``overtwisted" is a property that a given contact structure may satisfy. \newline

A first observation concerning our motivating question of embedding dynamics on Reeb fields is that the vector field $X$ cannot be arbitrary.

\begin{defin}
A vector field $X$ on $N$ is geodesible if there is some metric for which the orbits of $X$ are geodesics.
\end{defin}
When $X$ is of unit length for such a metric, we say that $X$ is geodesible of unit length. From now on, by geodesible we mean geodesible of unit length. A characterization of geodesible vector fields was given by Gluck in terms of differential forms: $X$ is geodesible if and only if there is some one-form $\beta$ such that $\beta(X)=1$ and $\iota_Xd\beta=0$. In particular, if a Reeb vector field $R$ defined by a form $\alpha$ on a contact manifold $M$ has some invariant submanifold $N$, then $R$ restricted to $N$ is geodesible. Indeed, if $X$ is the vector field $R$ restricted on $N$ and $i:N\hookrightarrow M$ is the inclusion of $N$ into $M$, then $i^*\alpha$ satisfies
\begin{equation}
\begin{cases}
i^*\alpha(X)=1 \\
\iota_X di^*\alpha=0
\end{cases}.
\end{equation} Note that $i^*\alpha$ is not necessarily a contact form, so that $X$ is not necessarily a Reeb field (in general, it is not even volume-preserving). However, it is always geodesible according to Gluck's characterization. \newline

Conversely, start with any geodesible (hence non-vanishing) vector field $X$ on a compact manifold $N$.
\begin{defin} \label{def:ReebEmb}
An embedding $e:(N,X) \hookrightarrow (M,\xi)$ is called a \emph{Reeb embedding} if there is a contact form $\alpha$ defining $\xi$ such that the associated Reeb field satisfies $e_*X=R|_{e(N)}$.
\end{defin}

The main theorem in \cite{CMPP1} gives sufficient conditions in terms of the codimension of an arbitrary smooth embedding to be isotopic to a Reeb embedding.
\begin{thm}[\cite{CMPP1}]\label{thm:flexReeb}
Let $e: (N,X) \hookrightarrow (M,\xi)$ be a smooth embedding of $N$ into a contact manifold $(M,\xi)$ where $X$ is a geodesible vector field on $N$. Assume that $\dim M \geq 3n+2$. Then $e$ is isotopic to a $C^0$-close Reeb embedding $\tilde e:(N,X)\hookrightarrow (M,\xi)$.
\end{thm}

\begin{rem}
If we impose the additional assumption that $(M,\xi)$ is an overtwisted contact manifold, then $\dim M \geq 3n$ is enough, although the Reeb embedding $\tilde e$ is not necessarily $C^0$ close to $e$ if $\dim M <3n+2$. In \cite{CMPP1}, parametric versions of the previous statement are also discussed.
\end{rem}

\begin{exa}
The existence of a Reeb embedding of any pair $(N,X)$ into some contact manifold is easy to establish, since there is a natural source of examples of such embeddings. Denote by $\beta$ the one-form such that $\beta(X)=1$ and $\iota_Xd\beta=0$. Gluck's characterization implies that there is a metric for which $X$ is of unit-length and its orbits are geodesics which satisfies $g(X,\cdot)=\beta$. Recall that the cotangent bundle $T^*N$ is equipped with the canonical Liouville one-form $\lambda_{std} \in \Omega^1(T^*N)$. Such one-form is characterized by the property that, given any one-form $\gamma$ on $N$, which can be understood as an embedding $\gamma:N\rightarrow T^*N$, we have $\gamma= \gamma^*\lambda_{std}$. For a given metric one can define the unit tangent bundle $STN$ defined fiberwise by $ST_pN=\{X\in T_pN \enspace | \enspace g_p(X,X)=1\}$. A standard property (see e.g.~\cite[Section 1.5]{Geig}) of $\lambda_{std}$ is that given the metric $g$ on $N$, it restricts on $ST^*N$ (the unit cotangent bundle) as a contact form $\lambda$ whose Reeb field is dual to the geodesic vector field on $STN$. In particular, the section $\beta$, seen as an embedding
$$ \beta: N \rightarrow ST^*N $$
satisfies $\beta^*\lambda=\beta$ and actually the Reeb field $R$ defined by $\lambda$ satisfies $\beta_* X=R$. Thus, it is a Reeb-embedding according to Definition \ref{def:ReebEmb}. This further motivates a systematic examination of Reeb-embeddings from a contact topology point of view, a study that leads to Theorem \ref{thm:flexReeb}.
\end{exa}

\begin{proof}[Sketch of the proof of Theorem \ref{thm:flexReeb}]

The proof of Theorem \ref{thm:flexReeb} follows the usual procedure of $h$-principle type results. We first define a ``formal" notion of Reeb embedding, which satisfies a property that is purely homotopic in terms of its differential. We then prove that, under certain conditions, any formal Reeb embedding is isotopic to a genuine Reeb embedding (i.e., they satisfy the $h$-principle). To conclude, we use obstruction theory to analyze the minimal codimension for which any smooth embedding is a formal Reeb embedding satisfying the conditions for the $h$-principle to apply. We will now sketch each of these steps of the proof, under the simplifying assumption that $M$ is overtwisted.

\paragraph{\textbf{Step 1: Iso-Reeb embeddings and extension lemma.}}

Let $X$ be a geodesible vector field on $N$, and denote by $\beta$ a one-form such that $\beta(X)=1$ and $\iota_Xd\beta=0$. We need to fix such a choice of one-form, and let $\eta:=\ker \beta$. Let $(M,\xi)$ be an overtwisted contact manifold with defining contact form $\alpha$, i.e., $\ker \alpha=\xi$.

With a slight abuse of notation, we will denote $\alpha \circ F_1$ for $\alpha(F_1(\cdot))$ and $d\alpha \circ F_1$ for $d\alpha(F_1(\cdot),F_1(\cdot))$. This is also denoted by ${F_1}^*\alpha$ and ${F_1}^*d\alpha$ in the discussion of ``generalized iso-contact immersions" in~\cite[Section 16.2]{hprinc}.

\begin{defin}\label{iso-Reeb}
An embedding $f:(N,X,\eta=\ker \beta)\rightarrow (M,\xi)$ is an \emph{iso-Reeb} embedding if $f_*(\eta)=\xi$.
\end{defin}

The corresponding formal notion is

\begin{defin}\label{fR} An embedding $f:(N,X,\eta)\rightarrow (M,\xi)$ is a \emph{formal iso-Reeb} embedding if there exists a homotopy of monomorphisms
$$ F_t:TN \longrightarrow TM,  $$
such that $F_t$ covers\footnote{We say that $F_t:TN \rightarrow TM$ covers $f:N\rightarrow M$ if the map between bases induced by $F_t$ is constantly equal to $f$.} $f$, $F_0=df$, $h_1 \alpha \circ F_1= \beta$ and $d\beta|_{\eta}= h_2 d\alpha \circ F_1|_{\eta} $ for some strictly positive functions $h_1$ and $h_2$ on $N$.
\end{defin}
Any (genuine) iso-Reeb embedding is clearly a formal iso-Reeb embedding, with $F_t$ constanly equal to $df$. Both conditions $h_1 \alpha \circ F_1= \beta$ and $d\beta|_{\eta}= h_2 d\alpha \circ F_1|_{\eta} $ have to be imposed, since $F_1$ does not commute with the exterior derivative in general (when $F_1$ is not holonomic). This formal notion of Reeb embedding is enough to obtain the main theorem for an overtwisted target contact manifold. For the most general case, an extra formal hypothesis needs to be imposed (confer \cite{CMPP1}). \newline

The following lemma by Inaba \cite{In} (see also~\cite{CMPP1}) provides the desired property of iso-Reeb embeddings.

\begin{lem}
Let $N$ be a submanifold of $(M,\xi)$, and denote by $i$ the inclusion map of $N$ into $M$. Let $\eta$ be the restriction $i^*\xi$. A nonvanishing vector field $X$ on $N$ can be extended to a Reeb field on all $M$ if and only if $X$ is transverse to $\eta$ and the flow of $X$ preserves $\eta$.
\end{lem}
An iso-Reeb embedding $f$ is in particular a Reeb embedding according to Definition \ref{def:ReebEmb}. The vector field $X$ is transverse to $\eta$ and preserves it if and only if there is a one form $\beta$ such that $\beta(X)=1$, $\iota_Xd\beta=0$ and $\ker \beta=\eta$. These are our hypotheses in the case of an iso-Reeb embedding, hence by the previous lemma there is a contact form whose Reeb field $R$ satisfies $f_*X=R$.

\paragraph{\textbf{Step 2: An $h$-principle via isocontact embeddings.}}

Our goal in this second step is to prove that any formal iso-Reeb embedding $e:(N,X,\eta) \rightarrow (M,\xi)$ into an overtwisted contact manifold, is homotopic through formal iso-Reeb embeddings to a genuine iso-Reeb embedding. This is tantamount to saying that iso-Reeb embeddings satisfy an existence $h$-principle. Other versions of the $h$-principle (parametric, relative to the domain, etc...) are discussed in \cite{CMPP1}. Recall that $\alpha$ is a defining contact form of $\xi$. The sketch of the argument is the following:
\begin{enumerate}
\item The embedding $e$ satisfies that $de(\eta)\subset TM|_N$, but $de(\eta)$ is not, in general, contained in $\ker \alpha=\xi$. We extend the homotopy $F_t$ and use it inversely to deform $\xi$ via an homotopy of symplectic vector bundles $(\xi_t,\omega_t)$ (defined over all $M$, but which is identically $(\xi,d\alpha)$ outside a neighborhood $U$ of $e(N)$) such that $(\xi_0,\omega_0)=(\xi,d\alpha)$, $(\xi_1,\omega_1)$ satisfies $de(\eta)\subset \xi_1$ and $\omega_1|_{\eta}=d\beta$ along $N$. The last condition is guaranteed, up to conformal transformation, by the formal iso-Reeb condition.
\item Using partitions of unity, the fact that $\omega_1$ is non degenerate on $\xi_1$, and that $\omega_1|_{\eta}=d\beta$,  it is now possible to make another deformation. We extend the homotopy $(\xi_t,\omega_t)$ to $t\in [1,2]$ such that $(\xi_2,\omega_2)$ is a contact structure in a smaller neighborhood $U'$ of $e(N)$ and still satisfies $de(\eta)\subset \xi_2$. In particular, we can achieve that $\omega_2=d\gamma$ for some one form $\gamma$ such that $\gamma$ satisfies $e^*\gamma=\beta$ (the form such that $\ker \beta=\eta$ and $\beta(X)=1$). The pair $(\xi_2,\omega_2)$ will not be a contact structure globally, since this small neighboorhod is a priori smaller than the neighborhood $U$ where $(\xi_1,\omega_1)$ was not anymore of contact type. Hence in some parts $U\setminus U'$, $\xi_2$ is not of contact type.
\item We will now reduce to a formal isocontact embedding (confer \cite[Section 12.3]{hprinc} for more details on such embeddings). We endow the neighborhood $U'$ with the contact structure $(\xi_2, \omega_2)$. We use the previous deformations $(\xi_t,\omega_t), t\in [0,2]$ defined on $U'$ to endow the trivial embedding $\hat e: U' \rightarrow M$ (defined as a neighborhood extension of the embedding $e$) with an homotopy of monomorphisms $G_t:TU' \rightarrow TM$ such that $G_0=d\hat e$, $G_1$ satisfies $\xi_2=G_1^{-1}(\xi)$ and the map induces a conformally symplectic map.
\item The map $\hat e$ is a formal isocontact embedding of codimension $0$ with open source manifold. The $h$-principle for such embeddings into overtwisted targets applies \cite[Corollary 1.4]{BEM}. We obtain an embedding $\tilde e: U'\rightarrow M$ (isotopic to $\hat e$ through formal isocontact embeddings) such that $d\tilde e$ satisfies $d\tilde e(\xi_2)=\xi$ and the map induces a conformally symplectic map. Since $(\xi_2,\omega_2)$ restricted to $N\subset U'$ corresponds to $(\eta,d\beta)$, we deduce that $\tilde e|_N$ is a genuine iso-Reeb embedding isotopic to $e=\hat e|_N$.
\end{enumerate}

\paragraph{\textbf{Step 3: Obstruction theory.}}

The final step of the proof consists in showing that for $\dim M\geq 3\dim N$, any smooth embedding $e:N\rightarrow (M,\xi)$ is a formal iso-Reeb embedding for any choice of $(X,\beta)$ where $X$ is a non-vanishing geodesible field and $\beta$ is a choice of one-form for which $\beta(X)=1$ and $\iota_Xd\beta=0$. We will assume the following lemma, confer \cite{CMPP1} for the details.
\begin{lem}\label{lem:symp}
Let $e:(N,X,\eta)\rightarrow (M,\xi)$ be an embedding such that there is an homotopy of monomorphisms $F_t:TN\rightarrow TM$ covering $e$ satisfying $F_0=de$ and $F_1(\eta)$ is an isotropic subbundle of $\xi$. Then $e$ is a formal iso-Reeb embedding.
\end{lem}
For $2m> \dim N$, it can be proved that there is a family of monomorphisms $H_t: TN \rightarrow TM$ such that $F_1(X) \pitchfork \xi$, and furthermore $F_1(\eta)\subset \xi$. The previous lemma shows that a sufficient condition for being a formal iso-Reeb embedding is that $F_1(\eta)$ can be deformed into an isotropic subbundle of $\xi$. Recall that $n$ denotes the dimension of $N$, hence $\eta$ has rank $n-1$. The manifold $M$ is of dimension $2m+1$ hence $\xi$ is of rank $2m$. Denote by $\operatorname{Gr}=\operatorname{Grass}(n-1,\mathbb{R}^{2m})$ the space of $(n-1)$-subspaces of $\mathbb{R}^m$. Similarly, denote by $\operatorname{Gr}_{\operatorname{is}}=\operatorname{Grass}_{\operatorname{is}}(n-1,\mathbb{R}^{2m})$ the space of isotropic subspaces of dimension $n-1$ in $\mathbb{R}^{2m}$ seen as $\mathbb{C}^m$. To find a path between $\eta$ and an isotropic subspace of $\xi$ over $N$, we need to find a global section of the bundle $E$ over $N$ whose fiber is
$$ P= \operatorname{Path}(\operatorname{Grass}(n-1,\mathbb{R}^{2m}),\operatorname{Grass}_{\operatorname{is}}(n-1,\mathbb{R}^{2m})), $$
i.e., the space of paths between any $(n-1)$-subspace and any isotropic $(n-1)$-subspace of $\mathbb{R}^{2m}$. On the other hand, we know that the homotopy groups of such a path space depend on the relative homotopy groups
$$ \pi_j (P) \cong \pi_{j+1}( \operatorname{Grass}(n-1,\mathbb{R}^{2m}), \operatorname{Grass}_{\operatorname{is}}(n-1,\mathbb{R}^{2m})).$$
We now use that
\begin{eqnarray*}
\operatorname{Gr} &  \cong  & \frac{SO(2m)}{SO(n-1) \times SO(2m-(n-1))}, \\
\operatorname{Gr}_{is} & \cong & \frac{U(m)}{SO(n-1)\times U(m-(n-1))}.
 \end{eqnarray*}
 Combining the exact sequence for relative pairs, the exact sequence for quotients, and using the stable range of the involved groups, we can show that
$$ 2m\geq 3n-1 \implies \pi_{j}(P)=0 \text{ for all } j\leq n-1\,. $$
Hence, if $\dim M\geq 3\dim N$, we can find a global section along $N$. Using this section and the previous family of monomorphisms, we find a family of isomorphisms $G_t:TN \rightarrow TM$ covering the smooth embedding $e$ such that $G_1(\eta)$ is an isotropic subbundle of $\xi$. Applying Lemma \ref{lem:symp}, we conclude that $e$ is a formal iso-Reeb embedding.

\paragraph{\textbf{Step 4: Conclusion.}} In step $3$ we showed that any smooth embedding is a formal iso-Reeb embedding for any pair $(N,X)$ embedded into a contact manifold $(M,\xi)$ such that $\dim M\geq 3\dim N$. Note that smooth embeddings in this context always exist by Whitney's embedding theorem. Under the assumption that $M$ is overtwisted, we can apply the $h$-principle proved in Step 2 and deduce that there is an iso-Reeb embedding $\tilde e$ isotopic to $e$.  Since an iso-Reeb embedding is in particular a Reeb embedding, we can find some contact form $\alpha$ defining $\xi$ whose Reeb field $R$ satisfies $\tilde e_*X=R|_{\tilde e(N)}$. This concludes the proof of the theorem.
\end{proof}

The previous theorem ``fixes" the target contact structure, which forces to take an embedding that is isotopic to the original smooth embedding $e:N\rightarrow (M,\xi)$. If we simply want to extend the vector field $X$ to a Reeb vector field, without fixing the ambient contact structure, then we can fix the embedding.

\begin{cor}\label{cor:fixedemb}
Let $X$ be a geodesible vector field on a compact manifold $N$. Let $e:N\rightarrow (M,\xi)$ be a smooth embedding into a contact manifold with $\dim M\geq 3\dim N+2$. Then there is a contact form $\alpha$ on $M$ whose Reeb field $R$ satisfies $e_*X=R|_{e(N)}$. The contact form $\alpha$ defines a contact structure contactomorphic to~$\xi$.
\end{cor}

\begin{proof}
It follows from Theorem \ref{thm:flexReeb} that there is a Reeb embedding $\tilde e$ (with respect to the contact structure $\xi$) isotopic to $e$. According to Definition \ref{def:ReebEmb}, there is a contact one-form $\alpha'$ defining $\xi$ such that the Reeb field $R'$ of $\alpha'$ satisfies $\tilde e_*X=R'|_{\tilde e(N)}$. Let $\varphi_t$ be an isotopy of $M$ such that $\varphi_1 \circ \tilde e=e$. Then $\alpha:=(\varphi_1^{-1})^*\alpha'$ is a contact one-form, defining a contact structure ${(\varphi_1)}_*\xi$, whose Reeb field $R={(\varphi_1)}_*R'$ satisfies
$$ e_* X= {(\varphi_1)}_* \circ \tilde e_* X= {(\varphi_1)}_* R'=R\,,$$
thus concluding the proof.
\end{proof}

\subsection{Applications to universality}

We are now ready to give some applications of Theorem \ref{thm:flexReeb}. The following concept is inspired by Tao's definition of Euler-extendibility in~\cite{T3} (albeit it is different in the sense that it is adapted to the context of stationary solutions of the Euler equations).

\begin{defin}\label{def1}
A non-autonomous time-periodic vector field $u_0(\cdot,t)$ on a compact manifold $N$ is \emph{Euler-extendible} if there exists an embedding $e:N\times \mathbb S^1 \rightarrow \mathbb S^n$ for some dimension $n>\text{dim }N+1$ (that only depends on the dimension of $N$), and an Eulerisable flow $u$ on $\mathbb S^n$, such that $e(N\times\mathbb S^1)$ is an invariant submanifold of $u$ and $e_*(u_0(\cdot,\theta)+\partial_\theta)=u|_{e(N\times \mathbb S^1)}$, $\theta\in\mathbb S^1$. If the non-autonomous field $u_0(\cdot,t)$ is not time-periodic, we say that it is Euler-extendible if there exists a proper embedding $e:N\times \mathbb R \rightarrow \mathbb R^n$ for some dimension $n>\text{dim }N+1$ (that only depends on the dimension of $N$), and an Eulerisable flow $u$ on $\mathbb R^n$, such that $e(N\times\mathbb R)$ is an invariant submanifold of $u$ and $e_*(u_0(\cdot,\theta)+\partial_\theta)=u|_{e(N\times\mathbb R)}$, $\theta\in\mathbb R$.
If any non-autonomous dynamics $u_0(\cdot,t)$ is Euler-extendible, we say that the stationary Euler flows are \emph{universal}.
\end{defin}

Roughly speaking, the extendibility of a non-autonomous dynamics implies that, in the appropriate local coordinates, $u_0$ describes the ``horizontal'' behavior of the integral curves of the extended vector field $u$. Observe that the original vector field $u_0$ is not assumed to be volume-preserving, although certainly $u$ will be. We introduce another definition for embeddability of discrete dynamics.

\begin{defin}\label{def:Euleremb}
We say that a (orientation-preserving) diffeomorphism $\phi:N\rightarrow N$ is \emph{Euler-embeddable} if there exists an Eulerisable field $u$ on $\mathbb S^n$ (for some $n$ that only depends on the dimension of $N$) with an invariant submanifold exhibiting a cross-section diffeomorphic to $N$ such that the first return map of $u$ at this cross section is conjugate to $\phi$.
\end{defin}

Two main corollaries of the previous construction can be expressed in terms of these two definitions.

\begin{cor}[\cite{CMPP1}]\label{cor:univ}
The stationary Euler flows are universal. Moreover, the dimension of the ambient manifold $\mathbb S^n$ or $\mathbb R^n$ is the smallest odd integer~$n\in\{3\text{ dim }N+5,3\text{ dim }N+6\}$. In the time-periodic case, the extended field $u$ is a steady Euler flow with a metric $g=g_0+\delta_P$, where $g_0$ is the canonical metric on $\mathbb S^n$ and $\delta_P$ is supported in a ball that contains the invariant submanifold $e(N\times\mathbb S^1)$.
\end{cor}
It is clear that the extension to an Euler flow $u$ is not unique, since Theorem \ref{thm:flexReeb} shows that iso-Reeb embeddings exist in abundance. Corollary \ref{cor:fixedemb}, via the correspondence Theorem \ref{thm:EGcorres}, illustrates the flexibility of steady Euler flows in the sense that \emph{any} fixed smooth embedding in high enough codimension can be realized as an invariant submanifold (with arbitrary induced geodesible dynamics) of a steady Euler flow. Our second corollary is expressed in terms of Definition \ref{def:Euleremb}.
\begin{cor}[\cite{CMPP1}]\label{cor:diffeo}
Let $N$ be a compact manifold and $\phi$ an orientation-preserving diffeomorphism on $N$. Then $\phi$ is Euler-embeddable in $\mathbb S^n,$ where $n$ is the smallest odd integer $n\in\{3\text{ dim }N+5,3\text{ dim }N+6\}$.
\end{cor}

As in Corollary \ref{cor:univ}, the metric can also be assumed to be the canonical one outside an embedding of the mapping torus of $N$ by $\phi$. This is ensured by applying Theorem \ref{thm:flexReeb} with a tight contact sphere as the target contact manifold. The dimensional bounds can be slightly improved if we use an overtwisted contact sphere as target manifold, as explained after the statement of Theorem \ref{thm:flexReeb}. In the following section, we shall introduce the concept of ``Turing complete" flows, which are flows that are universal in a computational sense. Using the fact that there are diffeomorphisms that simulate any Turing machine (see \cite{T1} for an example), and the fact that our construction via an $h$-principle is constructible (i.e., algorithmic), we obtain as a by-product that there is a Turing complete Euler flow on $\mathbb S^{17}$. In the next section, we will focus on this property and drastically improve the dimension of the ambient manifold.

\section{A Turing complete steady Euler flow on $\mathbb S^3$}\label{S.Turing}

In this section we review the construction of a Turing complete stationary Euler flow on a Riemannian three-sphere \cite{CMPP2}. We end up by proving a new result (Corollary \ref{cor:orbit}) on the existence of Reeb flows (and their Beltrami counterparts) with orbits whose periodicity is undecidable.

\subsection{Turing machines and symbolic dynamics}

A Turing machine is a mathematical model of a theoretical device manipulating a set of symbols on a tape following some specific rules. It receives, as input data, a sequence of symbols and, after a number of steps it might return as output another string of symbols. More concretely:

A Turing machine is defined via the following data:

\begin{itemize}
\item A finite set $Q$ of ``states'' including an initial state $q_0$ and a halting state $q_{halt}$.
\item A finite set $\Sigma$ which is the ``alphabet'' with cardinality at least two.
\item A transition function $\delta:Q\times \Sigma \longrightarrow Q\times \Sigma \times \{-1,0,1\}$.
\end{itemize}

We will denote by $q\in Q$ the current state, and by $t=(t_n)_{n\in \mathbb{Z}}\in \Sigma^\mathbb{Z}$ the current tape of the machine at a given step of the algorithm of the Turing machine. This gives a configuration $(q,t)$ of the machine. In particular, the space of all possible \emph{configurations} of a Turing machine is given by $\mathcal{P}:=Q\times \Sigma^\mathbb{Z}$. The algorithm works as follows, for a given input tape $t\in \Sigma^\mathbb{Z}$.

\begin{enumerate}
\item Set the current state $q$ as the initial state and the current tape $t$ as the input tape.
\item If the current state is $q_{halt}$ then halt the algorithm and return $t$ as output. Otherwise compute $\delta(q,t_0)=(q',t_0',\varepsilon)$, with $\varepsilon \in \{-1,0,1\}$.
\item Replace $q$ with $q'$ and $t_0$ with $t_0'$, obtaining a modified tape $\tilde t=(...t_{-1}.t_0't_1...)$.
\item Shift $\tilde t$ by $\varepsilon$, obtaining a new tape $t'$. The resulting configuration is $(q',t')$. Return to step $(2)$.
\end{enumerate}
Our convention is that $\varepsilon=1$ (resp. $\varepsilon=-1$) corresponds to the left shift (resp. the right shift). This algorithm (determined by the transition function $\delta$) induces a global transition function in the space of configurations $\Delta: Q\setminus \{q_{halt} \} \times \Sigma^{\mathbb{Z}} \rightarrow \mathcal{P}$, which sends a non-halting configuration in $\mathcal{P}$ to the configuration obtained after one step of the algorithm.

\begin{rem}
Without loss of generality, one can assume that the configurations of the machine are those pairs $(q,t)\in Q\times \Sigma^{\mathbb{Z}}$ for which only a finite number of symbols in $t$ are different from $0$ (also called the ``blank" symbol). We will not need this simplifying assumption in this section, although it is certainly useful in other constructions \cite{CMP2}.
\end{rem}

\paragraph{\textbf{The halting problem:}}

In computability theory, the halting problem is the problem of determining, from a description of an arbitrary computer program and an input, whether the program will finish running (halting state), or continue to run forever. Alan Turing proved in 1936 that a general algorithm to solve the halting problem for all possible program-input pairs cannot exist.
A key part of the proof is  the formulation of a mathematical definition of a computer and program, which is the previously introduced notion of Turing machine; the halting problem is undecidable for Turing machines.
The halting problem is historically important as it was one of the first problems to be proved undecidable.

\paragraph{\textbf{Turing machines and universality}}\label{prelim}

An Eulerisable field on a manifold $M$ is Turing complete if it can simulate any Turing machine. In fact, Turing machines can be simulated by dynamical systems in a large sense (a vector field, a diffeomorphism, a map, etc...).  Following \cite{T1}, we give a formal definition of such a ``simulation".

\begin{defin}\label{TC}
Let $X$ be a vector field on a manifold $M$. We say it is Turing complete if for any integer $k\geq 0$, given a Turing machine $T$, an input tape $t$, and a finite string $(t_{-k}^*,...,t_k^*)$ of symbols of the alphabet, there exist an explicitly constructible point $p\in M$ and an open set $U\subset M$ such that the trajectory of $X$ through $p$ intersects $U$ if and only if $T$ halts with an output tape whose positions $-k,...,k$ correspond to the symbols $t_{-k}^*,...,t_k^*$. A completely analogous definition holds for diffeomorphisms of $M$.
\end{defin}

\begin{rem}
In the construction explained in this Section, the point $p$ depends on $T$, the input and the finite string, while the open set $U$ is always the same. In other constructions of Turing complete flows \cite{T1,CMPP1,CMP1}, the point $p$ only depends on $T$ and the input, and the open set $U$ depends on the finite string of the output. In particular, for a fixed machine and input we construct a point $p$ and we can ``measure" a posteriori what is the output of the machine up to some precision by looking which open sets are intersected by the trajectory of the flow through $p$.
\end{rem}

\begin{rem}
One might as well avoid fixing a finite string of the output $(t_{-k}^*,...,t_k^*)$ and just require that the machine halts if and only if the trajectory through $p$ enters certain open set. As detailed in \cite[Lemma 5.5]{CMPP2}, the computational power is the same with this simplification.
\end{rem}

In 1991, Moore \cite{Mo} introduced the notion of generalized shift to be able to \emph{simulate any Turing machine}; a generalized shift is a map that acts on the space of infinite sequences on a given finite alphabet.

Let $A$ be an alphabet and $S\in A^\mathbb{Z}$ an infinite sequence. A generalized shift $\phi:A^\Z \rightarrow A^\Z$ is specified by two maps $F$ and $G$ which depend on a finite number of specified positions of the sequence in $A^\Z$. Denote by $D_F= \{i,...,i+r-1\}$ and $D_G=\{j,...,j+l-1\}$ the sets of positions on which $F$ and $G$ depend, respectively. These functions take a finite number of different values since they depend on a finite number of positions. The function $G$ modifies the sequence only at the positions indicated by $D_G$:
\begin{align*}
G:A^l &\longrightarrow A^l \\
(s_{j}...s_{j+l-1}) &\longmapsto (s_{j}'...s_{j+l-1}')
\end{align*}
Here $s_j...s_{j+l-1}$ are the symbols at the positions $j,...,j+l-1$ of an infinite sequence $S\in A^\mathbb{Z}$.

On the other hand, the function $F$ assigns to the finite subsequence $(s_{i},...,s_{i+r-1})$ of the infinite sequence $S\in A^\mathbb{Z}$ an integer:
$$ F:A^{r}\longrightarrow \mathbb{Z}. $$

The generalized shift $\phi:A^\Z \rightarrow A^\Z$ corresponding to $F$ and $G$ is defined as follows:
\begin{itemize}
\item Compute $F(S)$ and $G(S)$.
\item Modify $S$ changing the positions in $D_G$ by the function $G(S)$, obtaining a new sequence $S'$.
\item Shift $S'$ by $F(S)$ positions. That is, we obtain a new sequence $s''_n=s'_{n+F(S)}$ for all $n\in \Z$.
\end{itemize}
The sequence $S''$ is then $\phi(S)$.\newline

Given a Turing machine there is a generalized shift $\phi$ conjugate to it. Conjugation means that there is an injective map
$\varphi: \mathcal{P} \rightarrow A^\mathbb{Z}$ such that the global transition function of the Turing machine is given by
$\Delta= \varphi^{-1}\phi\varphi$. In fact, if the Turing machine is reversible, it can be shown that the generalized shift is bijective.

{\bf Key observation}: Generalized shifts are conjugate to maps of the \emph{square Cantor set} $C^2:=C\times C \subset I^2$, where $C$ is the (standard) \emph{Cantor ternary set} in the unit interval $I=[0,1]$.

{\bf Point assignment}: take $A=\{0,1\}$ (this can be assumed without loss of generality). Given $s=(...s_{-1}.s_0s_1...)\in A^\mathbb Z$, we can associate to it an \emph{explicitly constructible point} in the square Cantor set. We just express the coordinates of the assigned point in base $3$: the coordinate $y$ corresponds to the \emph{expansion} $(y_0,y_1,...)$ where $y_i=0$ if $s_i=0$ and $y_i=2$ if $s_i=1$. Analogously, the coordinate $x$ corresponds to the \emph{expansion} $(x_{1},x_{2},...)$ in base $3$ where $x_i=0$ if $s_{-i}=0$ and $x_i=2$ if $s_{-i}=1$.

Moore proved that any generalized shift is conjugate to the restriction on the square Cantor set of a piecewise linear map defined on blocks of the Cantor set in $I^2$. This map consists of finitely many area-preserving linear components. If the generalized shift is bijective, then the image blocks are pairwise disjoint. An example is depicted in Figure \ref{fig:Cantormap}.
 Each linear component is the composition of two linear maps: a \emph{translation} and a positive (or negative) power of the \emph{horseshoe map} (or the Baker's map).

 \begin{figure}[!ht]\label{fig:Cantormap}
\begin{center}
\includegraphics[scale=0.3]{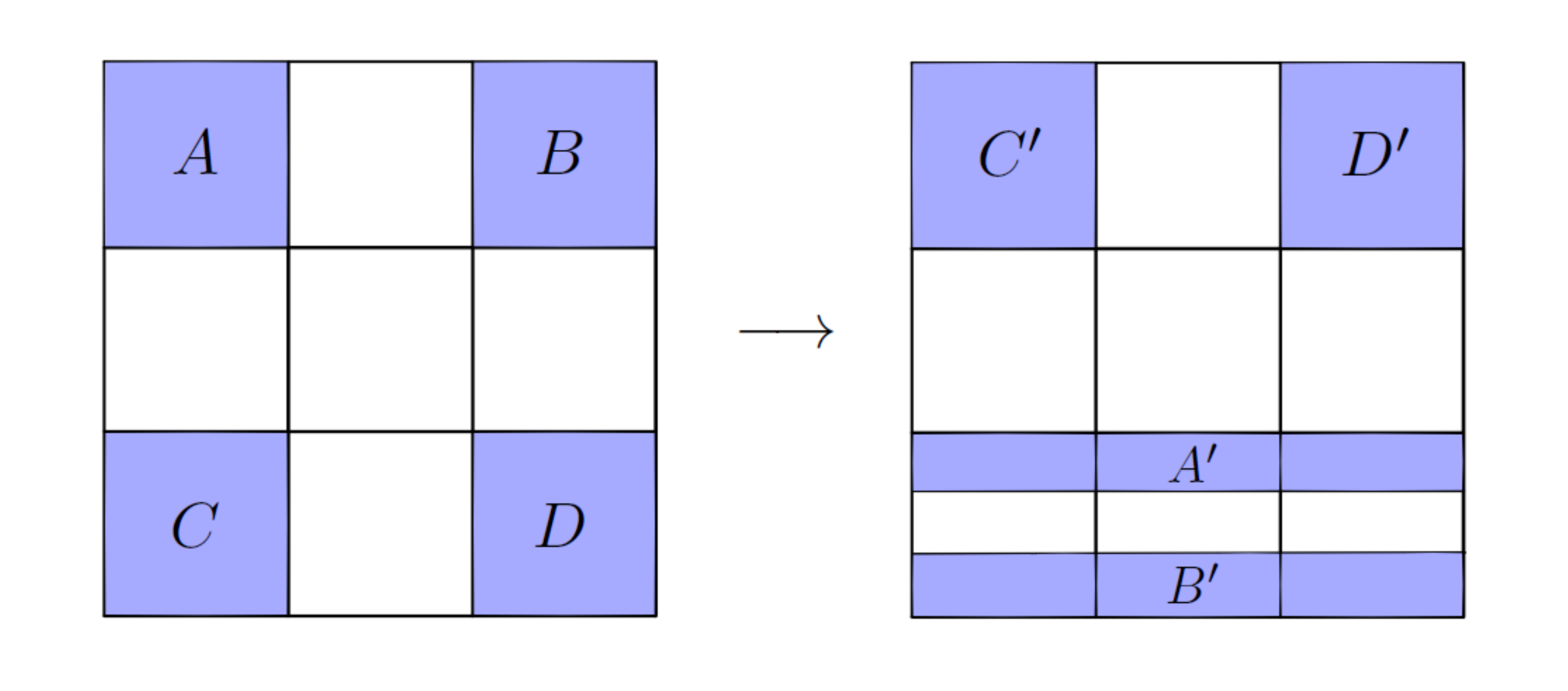}
\caption{Example of a map by blocks of the square Cantor set.}
\end{center}
\end{figure}


\subsection{Area-preserving maps and Turing complete Reeb flows}

In~\cite{CMPP2} we proved that any bijective generalized shift, understood as a map of the square Cantor set onto itself, can be extended as an area-preserving diffeomorphism of the disk which is the identity on the boundary. The proof of this result combines three ingredients: the aforementioned piecewise linear map defined on Cantor blocks, an explicit geometric construction using the homotopy extension property, and Moser's path method to ensure that the diffeomorphism that we obtain is area-preserving. The precise statement is the following:

\begin{prop}\label{prop:areaShift}
For each bijective generalized shift and its associated map of the square Cantor set $\phi$, there exists an area-preserving diffeomorphism of the disk $\varphi:D\rightarrow D$ which is the identity in a neighborhood of $\partial D$ and whose restriction to the square Cantor set is conjugate to $\phi$.
\end{prop}

Now the idea to construct a Turing complete Reeb flow is to take a Turing complete bijective generalized shift (which exists because there are universal Turing machines that are reversible). Proposition~\ref{prop:areaShift} hence implies the existence of a Turing complete area-preserving diffeomorphism of the disk which is the identity on the boundary, as detailed in \cite[Theorem 5.2]{CMPP2}. Using a suspension construction in contact geometry we can then show that any area-preserving diffeomorphism of the disk can be realized as the first-return map on a cross section of a Reeb flow on any contact three-manifold. In particular, taking the aforementioned Turing complete diffeomorphism we conclude the existence of Turing complete Reeb flows. More precisely:

\begin{thm}\label{thm:ReebDisk}
Let $(M,\xi)$ be a contact $3$-manifold and $\varphi:D \rightarrow D$ an area-preserving diffeomorphism of the disk which is the identity (in a neighborhood of) the boundary. Then there exists a defining contact form $\alpha$ whose associated Reeb vector field $R$ exhibits a Poincar\'e section with first return map conjugate to $\varphi$. In particular, there exists a Reeb field $R$ on $(M,\xi)$ which is Turing complete (in the sense of Definition~\ref{TC}).
\end{thm}

Combining Proposition~\ref{prop:areaShift}, Theorem~\ref{thm:ReebDisk} and the correspondence Theorem~\ref{thm:EGcorres} between Beltrami fields and Reeb flows, we obtain the desired result for stationary Euler flows.

\begin{cor}\label{cor:turing}
There exists an Eulerisable field $X$ on $\mathbb S^3$ that is Turing complete. The metric $g$ that makes $X$ a solution of the stationary Euler equations can be assumed to be the round metric in the complement of an embedded solid torus.
\end{cor}

The fact that the metric can be assumed to be the round one in the complement of an embedded solid torus needs some explanation. When applying Theorem~\ref{thm:ReebDisk}, we take as ambient manifold the standard contact sphere $(\mathbb S^3,\xi_{std})$. Then, the contact form whose Reeb field realizes a given area-preserving diffeomorphism of the disk as a Poincar\'e map can be chosen to coincide with the standard contact form $\alpha_{std}$ outside a solid torus. To conclude, one can  check that the metric associated to $\alpha$ via Theorem~\ref{thm:EGcorres} can be taken to be the round one whenever $\alpha$ coincides with $\alpha_{std}$.

\begin{rem}
The construction of a Turing complete Reeb flow in Theorem \ref{thm:ReebDisk} is obtained by choosing a particular reversible universal Turing machine and realizing its associated generalized shift as the first return map of the flow restricted to a square Cantor set on a Poincar\'e section (see Proposition \ref{prop:areaShift}). Had we chosen another reversible Turing machine (not necessarily universal),  its dynamics would have been induced in the square Cantor set via the first return map of a Reeb flow. We will use this observation in Corollary \ref{cor:undProp}.
\end{rem}

\subsection{Undecidable dynamical properties in Reeb dynamics}

In this subsection, we prove some new corollaries that follow from our construction in \cite{CMPP2}. A straightforward implication of Theorem~\ref{thm:ReebDisk} is the existence of certain phenomena of contact dynamics that are undecidable. Specifically, there is no algorithm to assure that a Reeb trajectory will pass through a certain region of space in finite time. The precise formulation of this result is the following:

\begin{cor}\label{cor:undecidableOrbits}
Let $R$ be a Turing complete Reeb flow on $(M,\xi)$. Then there exist an explicitly constructible compact set of points $K\subset M$ and an explicit open set $U\subset M$ such that it is undecidable to determine if the (positive) integral curves of $R$ through the points in $K$ will intersect the set $U$.
\end{cor}

A variation of our construction also allows us to construct a Reeb field $R$ for which there exist explicit points on $M$ such that the problem of determining if the orbit of $R$ through each of these points is closed is undecidable. The use of generalized shifts to find orbits whose periodicity is undecidable was also discussed by Moore in \cite{Mo}.

The aforementioned result follows from this auxiliary lemma.
\begin{lem}\label{keylemma}
There exists a Turing machine $T'$ such that:
\begin{enumerate}
\item It is reversible.
\item The image of the first component of the transition function $\delta$ does not contain $q_0$.
\item It satisfies the ``restart" property: if $T'$ halts with input $(q_0,t)$, then it halts with output $(q_{halt},t)$.
\item $T'$ is universal in the following sense: the halting of any Turing machine $T$ and input $c_0$ is equivalent to the halting of $T'$ for some explicit input (which depends on $T$ and $c_0$).
\end{enumerate}
\end{lem}

We are now ready to prove the undecidability of determining whether a trajectory is periodic or not:
\begin{cor}\label{cor:orbit}
Let $(M,\xi)$ be a three-dimensional contact manifold. Then there is a contact form $\alpha$ defining $\xi$ whose associated Reeb field $R$ satisfies that there are explicit points on $M$ for which determining whether the orbit through one of those points is periodic or not is an undecidable problem.
\end{cor}
\begin{proof}
Let $T=(Q,q_0,q_{halt},\Sigma,\delta)$ be a universal Turing machine as in Lemma \ref{keylemma}. We extend the transition function via $\delta(q_{halt},t)=(q_0,t,0)$, and construct a generalized shift $\phi$ conjugate to $T$ by a map $\varphi$. Then given any input $(q_0,t)$, the orbit of $\phi$ through $\varphi(q_0,t)$ is periodic if and only if $T$ halts with input $(q_0,t)$.

The map $\phi$ is bijective (since $T$ is reversible), and by Proposition \ref{prop:areaShift} we can find an area-preserving diffeomorphism of the disk $F:D \rightarrow D$ (which is the identity in a neighborhood of the boundary) whose restriction to the square Cantor set is conjugate to $\phi$. Using Proposition \ref{thm:ReebDisk}, we construct a contact form $\alpha$ defining $\xi$ whose Reeb flow has a cross-section with a first return map that is conjugate to $F$. It is then obvious that the orbit of the Reeb flow through a point representing an input of the Turing machine is periodic if and only if $T$ halts with such an input. The result then follows from the undecidability of the halting problem.
\end{proof}

Other special orbits can be constructed using the fact that the Turing machine is universal. For example, it is possible to construct an explicit point $p$ such that the orbit of the Reeb flow through $p$ is closed if and only if there is a counterexample to the Riemann hypothesis (using a discrete equivalent formulation \cite{T1}), and similarly with many other open problems in mathematics. This is achieved by constructing an initial input associated to a Turing machine which halts upon finding a counterexample.

Let us now give a proof of the auxiliary Lemma \ref{keylemma}.

\begin{proof}[Proof of Lemma \ref{keylemma}]
We construct the Turing machine with the ``restart" property by formalizing the construction sketched in \cite[Theorem 8]{KO}, even though there are several ways to construct it (another one is depicted in \cite[p 220]{Mo}). As explained in \cite[Section 6.1.2]{Mor}, we can find a reversible universal Turing machine $T=(Q,q_0,q_{halt},\Sigma,\delta)$ which satisfies property 2: the initial state cannot be reached from any other state. Let us construct a universal Turing machine $T'$ starting from $T$, which satisfies 1, 2 and 3.

This Turing machine is of the form $T'=(Q',q_0,q_{halt},\Sigma,\delta')$. The space of states $Q'$ is given by
$$Q'=(Q_0\times \{-1,+1\})\cup \{ q_0,q_{halt}\}\,,$$
where $Q_0:=Q\setminus\{q_0\}$. We basically take two copies of each state in $Q$ except for $q_0$, and add $q_0,q_{halt}$. The sign in $\{-1,+1\}$ denotes the ``direction" of the computation, a concept that will become clear in the construction. To simplify, for any state $q\in Q\setminus \{q_0,q_{halt}\}$, we denote by $q_+=q\times \{+1\} \in Q'$ and $q_-=q\times \{-1\} \in Q'$. The halting state of $T'$ is $q_{halt}$, even if there are two additional states $q_{halt}\times \{1\}$ and $q_{halt}\times \{-1\}$ that we denote by $q_{halt}^+$ and $q_{halt}^-$.

For any input of $T$, given by $(q_0,t)$, we associate the input $(q_0,t)$ of $T'$. For any pair of the form $(q_+,t)$ with $q\in Q\setminus \{q_0,q_{halt}\}$, we define the transition function of $T'$ exactly as the transition function $\delta$. To formalize this, we introduce the notation $(\tilde q, \tilde t_0, \varepsilon)=\delta(q,t_0)$. Then
$$ \delta'(q_+,t_0):=(\tilde q_+,\tilde t_0, \varepsilon). $$
This is always well defined since $\tilde q$ is never equal to $q_0$. Similarly, for the initial state $q_0$ we also use the notation $(\tilde q, \tilde t_0, \varepsilon)=\delta(q_0,t_0)$ and we define
$$ \delta'(q_0,t_0):=(\tilde q_+,\tilde t_0, \varepsilon). $$
When the machine reaches the state $q_{halt}^+$ (which happens when $T$ halts with that input), we reverse the computation by defining
\begin{equation}\label{eq:negative}
\delta'(q_{halt}^+,t_0):=(q_{halt}^-,t_0,0).
\end{equation}

The idea now is that instead of halting with the output of $T$, we swapped to a ``reverse the computations" phase to undo the computations. For the states $q_{halt}^-$ and $q_-$ with $q\not\in \{q_0,q_{halt}\}$, we define $T'$ as the inverse Turing machine: a step of $T'$ for a pair the form $(q_-,t_0)$ is given by $T^{-1}$. See for instance \cite[Section 5.1.4]{Mor} for the construction of the inverse machine $T^{-1}$, which is also reversible. Denote by $\delta^{-1}$ the transition function of $T^{-1}$; notice that $\delta^{-1}$ is not defined on the state $q_0$ by property 2. Then, for $q\in Q\setminus \{q_0,q_{halt}\}$, if we set $\delta^{-1}(q,t_0)=(\tilde q, \tilde t_0,\varepsilon)$, we define
$$ \delta'(q_-,t_0):= (\tilde q_{-},\tilde t_0, \varepsilon), \text{ if } \tilde q\neq q_0. $$
If $\delta^{-1}(q,t_0)=(q_0,\tilde t_0,\varepsilon)$, it means that we have returned to the input configuration so we can define instead:
\begin{equation}
 \delta'(q_-,t_0):=(q_{halt},\tilde t_0,\varepsilon)\,.
\end{equation}
Similarly for $q_{halt}^-$, if $\delta^{-1}(q_{halt},t_0)=(\tilde q, \tilde t_0,\varepsilon)$, we define
$$ \delta'(q_{halt}^-,t_0)=(\tilde q_-, \tilde t_0,\varepsilon) \text{ if } \tilde q \neq q_0 $$
and if $\tilde q=q_0$ then
\begin{equation}\label{eq:inithalt}
\delta'(q_{halt}^-,t_0)=(q_{halt}, \tilde t_0,\varepsilon).
\end{equation}
Notice that the image state $\tilde q$ via $\delta^{-1}$ cannot be $q_{halt}$ because the transition function $\delta$ is not defined when $q=q_{halt}$.

The global transition function of $T'$ on configurations with states $q_0, q_+$ coincides with the global transition function of $T$, where $q_{halt}^+$ is identified with the halting state of $T$. Accordingly, it is injective there. Similarly, the global transition function on configurations with states $q_-$ and $q_{halt}$ coincides with that of $T^{-1}$, where $q_{halt}$ is identified with the halting state of $T'$ and $q_{halt}^-$ is identified with the initial state of $T'$. So it is also injective there. Each configuration with state $q_{halt}^+$ is sent to the same configuration with state $q_{halt}^-$ in a trivial injective way. Summarizing, the global transition function of $T'$ is injective everywhere so $T'$ is reversible\\

The machine $T'$ satisfies 2, since $q_0$ cannot be reached from $\delta$, and in our construction we attain $q_{halt}$ instead of $q_0$ when $\delta^{-1}$ is applied according to equation \eqref{eq:inithalt}. The machine is universal since its halting is equivalent to the halting of $T$. Indeed, observe that the states of the form $q_-$ in $T'$ can only be reached if $T$ halted, and $q_{halt}$ can only be reached through negative states. This shows that if $T$ does not halt with input $(q_0,t)$ then $T'$ does not halt. On the other hand if $T$ halts, $T'$ will eventually reach a negative state, reverse the computation and reach $q_{halt}$. In fact, $T$ halts with input $(q_0,t)$ if and only if $T'$ halts with the same input. This shows that $T'$ is universal.

Property $3$ is also satisfied by construction. Whenever $T'$ halts with input $(q_0,t)$, it will reach a $q_{halt}^+$, then $q_{halt}^-$ and reverse the computation to halt with configuration $(q_{halt},t)$.
\end{proof}

\begin{rem}\label{R:restart}
Since any Turing machine can be simulated by a reversible Turing machine that satisfies property $2$ (see e.g.~\cite[Section 6.1.2]{Mor}), the construction presented in the proof of Lemma~\ref{keylemma} allows one to start from any reversible Turing machine $T$, obtaining a reversible Turing machine $T'$ which halts on the same inputs than $T$ and has the ``restart" property. In particular, any undecidable property associated to the inputs of $T$ that halt is inherited by the inputs of $T'$.
\end{rem}

Finally, we can mention a corollary which serves as a sample of dynamical properties of Reeb flows which simulate Turing machines that can be easily shown to be undecidable. Such undecidable properties are inherent to Turing machines and their associated generalized shifts \cite[Theorem 10]{Mo}. A key ingredient is Rice's theorem in computability theory, which in particular shows that non-trivial questions about the set of inputs for which the Turing machine halts are undecidable~\cite{Ri}. The following result is then a straightforward consequence of the previous Remark, the Remark after Corollary \ref{cor:turing}, and the existence of reversible Turing machines for which determining if the set of inputs that halt is dense (in the set of all inputs), has cardinality at least $k\geq 0$, etc. is undecidable.

\begin{cor}\label{cor:undProp}
Let $(M,\xi)$ be a three-dimensional contact manifold. Then there is a contact form $\alpha$ defining $\xi$ and a compact set $K\subset M$ invariant by the associated Reeb field $R$ for which the following questions on the dynamics of $R$ are undecidable (we remark that $\alpha$ depends on each question):
\begin{itemize}
\item Are there at least $k\geq 0$ periodic orbits on $K$?
\item Is the set of periodic orbits dense in $K$?
\item For a given $\mu>0$, is the set of periodic orbits on $K$ of measure greater than $\mu$?
\end{itemize}
\end{cor}
In the previous corollary, the set $K$ is simply the union of orbits which intersect the points associated to inputs (these points lie on a finite union of blocks of the square Cantor set, see \cite{CMPP2}).


\section{Time dependent solutions of Euler and Navier-Stokes}\label{S.future}

In the previous sections we have focused on stationary solutions to the Euler equations, first in high dimensions as a consequence of a new $h$-principle for Reeb embeddings, and then in dimension three using the power of symbolic dynamics. However, recall that the original motivation in \cite{T2,T3,TNat} was to find a Turing complete time-dependent solution. The time-dependent Euler equations on a Riemannian manifold $(M,g)$ define a dynamical system on the space of volume-preserving vector fields of the ambient manifold $\mathfrak X^\infty_{vol}(M)$. The following definition of Turing completeness is adapted to this context by analogy with Definition \ref{TC}.

\begin{defin}
Let $(M,g)$ be a Riemannian manifold. The Euler equations on $(M,g)$ are Turing complete if the following property is satisfied. For any integer $k\geq 0$, given a Turing machine $T$, an input tape $t$, and a finite string $(t_{-k}^*,...,t_k^*)$ of symbols of the alphabet, there exist an explicitly constructible vector field
$X_0\in \mathfrak X^\infty_{vol}(M)$ and a constructible open set $U\subset \mathfrak X^\infty_{vol}(M)$ such that the solution to the Euler equations with initial datum $X_0$ is smooth for all time and intersects $U$ if and only if $T$ halts with an output tape whose positions $-k,...,k$ correspond to the symbols $t_{-k}^*,...,t_k^*$.
\end{defin}

In our recent article \cite{CMP2}, we use a remarkable embedding theorem by Torres de Lizaur \cite{TdL} (building on a previous embedding theorem into time-dependent Euler flows by Tao~\cite{T2}) and the construction of Turing complete polynomial non-autonomous ODEs \cite{GCB}, to obtain Turing complete time-dependent solutions to the Euler equations:

\begin{thm}[\cite{CMP2}]\label{thm:timedependent}
There exists a (constructible) compact Riemannian manifold $(M,g)$ such that the Euler equations on $(M,g)$ are Turing complete.  In particular, the problem of determining whether a certain solution to the Euler equations with initial datum $X_0$ will reach a certain open set $U\subset\mathfrak X^\infty_{vol}(M)$ is undecidable.
\end{thm}

This solves the question of the Turing universality of the time-dependent Euler equations in high dimensions with general Riemannian metrics.

We finish this article presenting an application of Corollary~\ref{cor:turing} in the context of the Navier-Stokes equations (following~\cite{CMPP1}). These equations describe the dynamics of an incompressible fluid flow with viscosity. On a Riemannian $3$-manifold $(M,g)$ they read as~\cite{BSM12}
\begin{equation}\label{eq:NS}
\begin{cases}
\frac{\partial}{\partial t}X+ \nabla_X X-\nu \Delta X=-\nabla p\,,\\
\operatorname{div} X=0\,, \\
X(t=0)=X_0\,,
\end{cases}
\end{equation}
where $\nu>0$ is the viscosity. In what follows, the differential operators are computed with respect to the metric $g$, and $\Delta$ stands for the Hodge Laplacian (whose action on a vector field is defined as $\Delta X:=(\Delta X^\flat)^\sharp$).

Let us analyze what happens with the solution $X(t)$ when we take the Turing complete vector field $X_0$ constructed in Corollary~\ref{cor:turing} as initial condition (for the Navier-Stokes equations with the metric $g$ that makes $X_0$ a stationary Euler flow). Specifically, using that $\operatorname{curl}_g(X_0)=X_0$, the solution to Equation~\eqref{eq:NS} with initial datum $X(t=0)=M X_0$, $M>0$ a real constant, is easily seen to be
\begin{equation}
\begin{cases}
X(\cdot,t)=MX_0(\cdot)e^{-\nu t}\,, \\
p(\cdot,t)= c_0 - \frac{1}{2}M^2e^{-2\nu t} g(X_0,X_0)\,,
\end{cases}
\end{equation}
for any constant $c_0$. The integral curves (fluid particle paths) of the non-autonomous field $X$ solve the ODE
$$\frac{dx(t)}{dt}=M e^{-\nu t}X_0(x(t))\,.$$
Accordingly, reparametrizing the time as
$$\tau(t):=\frac{M}{\nu}(1-e^{-\nu t})\,,$$
we show that the solution $x(t)$ can be written in terms of the solution $y(\tau)$ of the ODE
\[
\frac{dy(\tau)}{d\tau}=X_0(y(\tau))\,,
\]
as
\[
x(t)=y(\tau(t))\,.
\]

When $t\rightarrow \infty$ the new reparametrized ``time'' $\tau$ tends to $\frac{M}{\nu}$, and hence the integral curve $x(t)$ of the solution to the Navier-Stokes equations travels the orbit of $X_0$ just for the time interval $\tau\in [0,\frac{M}{\nu})$. In particular, the flow of the solution $X$ only simulates a finite number of steps of a given Turing machine, so we cannot deduce the Turing completeness of the Navier-Stokes equations using the vector field $MX_0$ as initial condition. More number of steps of a Turing machine can be simulated if $\nu\to 0$ (the vanishing viscosity limit) or $M\to \infty$ (the $L^2$ norm of the initial datum blows up). For example, to obtain a universal Turing simulation we can take a family $\{M_k X_0\}_{k\in\mathbb N}$ of initial data for the Navier-Stokes equations, where $M_k\rightarrow \infty$ is a sequence of positive numbers. The energy ($L^2$ norm) of this family is not uniformly bounded, so it remains as a challenging open problem to know if there exists an initial datum of finite energy that gives rise to a Turing complete solution of the Navier-Stokes equations.

\begin{acknowledgments}
Robert Cardona acknowledges financial support from the Spanish Ministry of Economy and Competitiveness, through the Mar\'ia de Maeztu Programme for Units of Excellence in R\& D (MDM-2014-0445) via an FPI grant. Robert Cardona and Eva Miranda are partially supported by the grants PID2019-103849GB-I00 / AEI / 10.13039/501100011033 and the AGAUR grant 2017SGR932. Eva Miranda is supported by the Catalan Institution for Research and Advanced Studies via an ICREA Academia Prize 2016. Daniel Peralta-Salas is supported by the grants MTM PID2019-106715GB-C21 (MICINN) and Europa Excelencia EUR2019-103821 (MCIU). This work was partially supported by the ICMAT--Severo Ochoa grant CEX2019-000904-S.
\end{acknowledgments}

\end{document}